\documentclass[reqno]{amsart}
\usepackage{amsmath}
\usepackage{graphicx, color}
\usepackage{chngcntr}
\usepackage{tikz}
\usepackage{amsmath,amsfonts, amssymb}
\usepackage{wrapfig}
\usepackage{chngcntr}
\usepackage{paralist}
\usepackage{graphics} 
  \usepackage{epsfig}
 \usepackage[colorlinks=true]{hyperref}
 \usepackage{graphicx}  \usepackage{epstopdf}
\hypersetup{urlcolor=blue, citecolor=red}

\newtheorem{theorem}{Theorem}[section]

\newtheorem{lemma}[theorem]{Lemma}

\theoremstyle{definition}
\newtheorem{definition}[theorem]{Definition}
\newtheorem{remark}{Remark}

\title[Propagation of  stretched exponential moments]
      {Propagation of stretched exponential moments for the Kac equation and Boltzmann equation with Maxwell molecules}

\author[Milana Pavi\'c-\v Coli\'{c} and Maja Taskovi\'c]{}

 \email{milana.pavic@dmi.uns.ac.rs}
 \email{taskovic@math.upenn.edu}

\begin{document}
\maketitle

\centerline{\scshape Milana Pavi\'c-\v Coli\'{c}}
\medskip
{\footnotesize
  \centerline{Department of Mathematics and Informatics}
  \centerline{Faculty of Sciences, University of Novi Sad}
 \centerline{Trg Dositeja Obradovi\'ca 4, 21000 Novi Sad, Serbia}
}

\medskip

\centerline{\scshape Maja Taskovi\'{c}}
\medskip
{\footnotesize
 \centerline{Department of Mathematics}
 \centerline{University of Pennsylvania}
 \centerline{David Rittenhouse Lab.}
 \centerline{209 South 33rd Street, Philadelphia, PA 19104}

}

\bigskip

\begin{abstract}
We study the spatially homogeneous Boltzmann equation for Ma\-xwell molecules, and its $1$-dimensional model, the Kac equation. We prove propagation in time of stretched exponential moments of their weak solutions, both for the angular cutoff and the angular non-cutoff case. The order of the stretched exponential moments in question depends on the singularity rate of the angular kernel of the Boltzmann and the Kac equation. One of the main tools we use are Mittag-Leffler moments, which generalize the exponential ones.
\end{abstract}

\section{Introduction}

In this paper we study exponential tails (exponentially weighted $L^1$ norms) of weak solutions to the Kac equation and the spatially homogeneous Boltzmann equation for Maxwell molecules. We show propagation in time of such tails, both in the so-called cutoff and the non-cutoff case.

Both the Kac equation and the Boltzmann equation model the evolution of a probability distribution of particles inside a gas interacting via binary collisions. The models we consider are spatially homogeneous, which means that the probability distribution $f(t,v)$ depends only on time $t$, velocity $v$, but not of the spatial variable $x$. The Kac equation is a model for a 1-dimensional spatially homogeneous gas in which collisions conserve the mass and the energy, but not the momentum. On the other hand, the spatially homogeneous Boltzmann equation describes a gas in a $d$-dimensional space, with $d\geq2$, in which particles collisions are elastic, meaning they conserve the mass, momentum and energy.

The probability distribution function $f(t,v)$, for time $t\in \mathbb{R}^+$ and velocity $v\in \mathbb{R}^d$ (with $d=1$ for the Kac equation and $d\geq 2$ for the Boltzmann equation), changes due to the free transport and collisions. In the case of the spatially homogeneous Kac equation its evolution is modeled by the following equation
\begin{equation}\label{K}
\partial_t f(t,v) = \int_{\mathbb{R}} \int_{-\pi}^{\pi} \left( f' \, f'_* - f \, f_* \right) b_K(|\theta|) \, \mathrm{d} \theta \, \mathrm{d} v_*.
\end{equation}
The spatially homogeneous Boltzmann equation on the other hand reads
\begin{equation}\label{B}
\partial_t f(t,v) = \int_{\mathbb{R}^d} \int_{S^{d-1}} \left( f' \, f'_* - f \, f_* \right) |v-v_*|^\gamma \, b_B\left(\tfrac{v-v_*}{|v-v_*|}\cdot \sigma\right) \, \mathrm{d} \sigma \, \mathrm{d} v_*,
\end{equation}
which in the case of Maxwell molecules ($\gamma=0$) reduces to
\begin{equation}\label{MM}
\partial_t f(t,v) = \int_{\mathbb{R}^d} \int_{S^{d-1}} \left( f' \, f'_* - f \, f_* \right) \, b_B\left(\tfrac{v-v_*}{|v-v_*|}\cdot \sigma\right) \, \mathrm{d} \sigma \, \mathrm{d} v_*.
\end{equation}

Details about the notation employed are contained in Section \ref{sec-k} for the Kac equation and in Section \ref{sec-b} for the Boltzmann equation. For now we only remark that for both equations we consider angular kernels $b_K$ and $b_B$ that may of may not be integrable. When the angular singularity is non-integrable, our results depend on the singularity rate of the kernels.

The Kac equation \eqref{K} and the corresponding Boltzmann equation for Maxwell molecules \eqref{MM} share many properties (one notable difference is that the Kac equation does not conserve the momentum). In particular, both equations propagate polynomial and exponential moments, whose definitions we now recall.
\begin{definition} The polynomial moment of order $q$ of
the distribution function $f$ is defined by
\begin{equation}\label{polynomial moment}
m_q(t): = \int_{\mathbb{R}^d} f (t,v) \langle v \rangle^q \, \mathrm{d} v.
\end{equation}
\end{definition}
\begin{definition} The stretched exponential moment of order $s$ and rate $\alpha$ of
the distribution function $f$ is defined by
\begin{equation}\label{exp moment}
M_{\alpha,s}(t): = \int_{\mathbb{R}^d} f (t,v) e^{\alpha \langle v \rangle^s} \, \mathrm{d} v, \quad \alpha>0.
\end{equation}
In this paper, the special case when $s=2$ is referred to as the Maxwellian moment.
\end{definition}
When $f$ solves the Kac equation, the dimension $d$ in these formulas is one. We also remark that we use the following notation $\langle x \rangle:= \sqrt{1+  x_1^2 + \dots x_d^2}$, for any $x=(x_1,\dots, x_d)\in \mathbb{R}^d$, $d\geq 1$. The results presented in this paper are also valid when the moments are defined with absolute values $|v|$ in place of $\langle v \rangle$.\\

In the case of the Kac equation, the study of stretched exponential moments goes back to \cite{de93}. There, the constant angular kernel is considered, and the propagation of  stretched exponential moments of orders $s=1$ and $s=2$ is proved.

For the Boltzmann equation, propagation of Maxwellian moments was proved in the case of Maxwell molecules $\gamma=0$ in \cite{bo84, bo88} and recently in \cite{boga17} via Fourier transform techniques.  The theory was later extended to hard potentials $\gamma\in(0,1]$  in the context of Maxwellian moments in \cite{bo97, bogapa04, gapavi09, boga17}, and in the context of stretched exponential moments in \cite{mo06, alcagamo13}. Finally, the propagation of stretched exponential moments for the non-integrable angular kernels are studied in  \cite{lumo12, algapata15}.\\

In this paper, we generalize results of \cite{de93} to include more general orders of stretched exponentials, namely $s\in(0,2]$. The angular kernels that we study  are more general and may or may not be integrable. In the case of non-integrable angular kernels, the singularity rate affects the order of moments that propagate in time. In addition, we apply same technique to prove propagation of  stretched exponential moments for the Boltzmann equation with $\gamma =0$, thus extending the result of \cite{algapata15}.

We point out that the method we employ in this paper differs from the approach in \cite{de93}. Elegant calculations for exponential moments \eqref{exp moment} of order $s=1$ and $s=2$ in \cite{de93} are done directly at the level of exponential moments. In this manuscript, we take a different route and express exponential moments as infinite sums of polynomial moments and then strive to show that such infinite sums are finite. Such approach has been first developed in the context of the Boltzmann equation in \cite{bo97}, where the following fundamental relation was noted
\begin{align} \label{tt1}
 M_{\alpha,s}(t) = \int_{\mathbb{R}^d} f(t,v) \sum_{q=0}^\infty \frac{\alpha^q \langle v\rangle^{sq}}{q!}
	\; = \; \sum_{q=0}^\infty \frac{\alpha^q \; m_{sq}(t)}{\Gamma(q+1)}.
\end{align}
Finiteness of such sums can be studied by proving term-by-term geometric decay, or by showing that partial sums are uniformly bounded.

Our proof is inspired by the works \cite{alcagamo13, algapata15}, where the partial sum approach is developed. Moreover, motivated by \cite{algapata15}, we exploit the notion of Mittag-Leffler moments, which serve as a generalization of stretched exponential moments and which are very flexible for the calculations at hand. We recall the definition and the motivation for Mittag-Leffler moments in Section \ref{sec-ml}.

The paper is organized as follows. A brief review of the Kac equation in provided in Section \ref{sec-k}, while the review of the Boltzmann equation is contained in Section \ref{sec-b}. In Section \ref{sec-main} we state our main result. Section \ref{sec-ml} recalls the notion of Mittag-Leffler functions and moments, one of the main tools in the proof of the main theorem. Section \ref{sec-lem} contains another key tool - an angular averaging lemma with cancellation. In Section \ref{sec-pol}, the angular averaging lemma is used to derive  differential inequalities satisfied by polynomial moments of the solution to the Cauchy problem under the consideration. Finally, in Section \ref{sec-pf1} we provide the proof of the main theorem. The Appendix lists auxiliary Lemmas.

\section{The spatially homogeneous Kac equation}\label{sec-k}

The Kac model statistically describes the state of the gas in one dimension. The main object is the distribution function $f(t,x,v)\geq 0$ which depends on time $t\geq 0$, space position $x \in \mathbb{R}$ and velocity $v \in \mathbb{R}$, and which changes in time due to the free transport and collisions between gas particles. Assuming that collisions are binary and that they conserve mass and energy, but not momentum, the evolution of the distribution function is determined by the Kac equation.

In this paper we assume that the distribution function does not depend on the space position $x$, i.e. $f:=f(t,v)$. In that case, $f$ satisfies the spatially homogeneous Kac equation
\begin{equation}\label{Kac}
\partial_t f(t,v) = K(f,f)(t,v),
\end{equation}
where the collision operator $K(f,f)$ is defined by
\begin{equation}\label{K(f)}
K(f,f)(t,v)= \int_{\mathbb{R}} \int_{-\pi}^{\pi} \left( f' \, f'_* - f \, f_* \right)b_K(|\theta|) \, \mathrm{d} \theta \, \mathrm{d} v_*,
\end{equation}
with the standard abbreviations $f_*:=f(t,v_*)$, $f':=f(t,v')$, $f'_*:=f(t,v'_*)$.

The velocities    $v', v'_*$ and $v, v_*$ denote the pre and  post-collisional velocities for the pair of  colliding particles, respectively. A collision conserves the  energy of the two particles
\begin{equation*}\label{CL kinetic energy micro}
v'^{2} + v_*'^{2} = v^2 + v^2_*,
\end{equation*}
so by introducing  a parameter $\theta \in [-\pi,\pi]$, the collision rules read
\begin{align}\label{velocities}
v' & = v \, \cos \theta - v_* \sin \theta,\\
v'_* & = v \sin \theta + v_* \cos \theta. \nonumber
\end{align}
Note that a 2-dimensional vector $(v',v'_*)$ can be viewed as a rotation of the 2-dimensional vector $(v,v_*)$ by the angle $\theta$.

In this paper we assume that the angular kernel $b_K(|\theta|)\geq 0$  satisfies the following assumption
\begin{equation}\label{cross section}
\int_{-\pi}^\pi b_K(|\theta|) \, \sin^\kappa\theta \, \mathrm{d} \theta < \infty, \ \text{for some} \ \kappa \in [0,2].
\end{equation}
The case $\kappa = 0 $ corresponds to the so called Grad's cutoff case, when the angular kernel is integrable on $[-\pi,\pi]$. Otherwise, when $\kappa$ is strictly positive, i.e. the non-cutoff case, $b_K(|\theta|)$ is allowed to have $\kappa$ more degrees of singularity at $\theta=0$.

\subsection{Weak formulation of the collision operator}
Since the Jacobian of the transformation \eqref{velocities} is unit, for a test function $\phi(v)$, the weak formulation of the collision operator $K(f,f)$ reads
\begin{multline}\label{weak formulation K}
\int_{\mathbb{R}}  K(f,f) \, \phi(v) \, \mathrm{d} v \\ = \frac{1}{2} \int_{\mathbb{R}}\int_{\mathbb{R}} \int_{-\pi}^{\pi} f f_* \left( \phi(v') + \phi(v'_*) - \phi(v) - \phi(v_*) \right) b_K(|\theta|) \, \mathrm{d} \theta \, \mathrm{d} v_* \, \mathrm{d} v.
\end{multline}

\subsection{Weak solutions of the Kac equation} We recall the definition of a  weak solution to the Cauchy problem for the Kac equation
\begin{equation}\label{Cauchy1}
\left\{\begin{split} \partial_t f(t,v)&=K(f,f)(t,v) \quad t\in \mathbb{R}_+, \ v \in \mathbb{R},\\f(0,v)&=f_0(v), \end{split}\right.
\end{equation}
whose existence was proved in \cite{de93} for the cutoff case, i.e. $\kappa=0$, and  in \cite{de95} for the non-cutoff case, i.e. $\kappa \in (0,2]$.

\begin{definition}\label{def-weak sol for Kac}
Let $f_0\geq 0$ be a function defined on $\mathbb{R}^d$ with finite mass, energy and entropy, i.e.
\begin{equation}\label{Kac initial data}
\int_{\mathbb{R}} f_0(v) \left( \langle v \rangle^2 + \left| \log f_0 (v) \right| \right) \mathrm{d}v <  \infty.
\end{equation}
Then we say $f \geq 0$ is a {\it weak solution} to the Cauchy problem \eqref{Cauchy1} with $K(f,f)$ given by \eqref{K(f)}   if $f(t,v) \in L^{\infty}\left( \left[0, + \infty \right); L^1_2   \right)$, and
for all test functions $\phi \in W^{2,\infty}(\mathbb{R}_v)$ we have
\begin{equation*}
\partial_t \int_{\mathbb{R}} f \phi(v) \mathrm{d}v = \int_{\mathbb{R}} \int_{\mathbb{R}} K^\phi(v,v_*) f f_* \mathrm{d}v\mathrm{d}v_*,
\end{equation*}
where
\begin{equation*}
K^\phi(v,v_*) = \int_{-\pi}^\pi \left( \phi(v') - \phi(v) \right) b_K(|\theta|) \mathrm{d}\theta.
\end{equation*}
\end{definition}
For these solutions conservation of mass holds
\begin{equation}\label{CL}
\int_{\mathbb{R}} f(t,v) \, \mathrm{d} v = \int_{\mathbb{R}} f_0(v) \,  \mathrm{d} v.
\end{equation}
while the energy decreases in time. However, the energy is conserved, that is
\begin{equation*}
\int_{\mathbb{R}} f(t,v) \, v^2 \, \mathrm{d} v = \int_{\mathbb{R}} f_0(v) \, v^2 \, \mathrm{d} v,
\end{equation*}
if  there exists $C>0$ such that
$$
\int_\mathbb{R} f_0 (v) \left(1+|v|^{2p}\right) \, \mathrm{d}v <C,
$$
for some $p\geq2$. For details, see \cite{de95}.

\section{The spatially homogeneous Boltzmann equation}\label{sec-b}

The state of gas particles which at a time $t\in \mathbb{R}^+$ have a  position $x\in\mathbb{R}^d$ and velocity $v\in \mathbb{R}^d$, $d\geq2$, is statistically described by  the distribution function $f=f(t,x,v) \geq 0$. The evolution of such distribution function is modeled by the Boltzmann equation, which takes into account the effects of the free transport and collisions on $f$.
The collisions are assumed to be binary and elastic, that is, they conserve mass, momentum and energy for any pair of colliding particles.

When the distribution function is independent of the spatial variable $x$, that is $f:=f(t,v)\geq 0$, which is the so called spatially homogeneous case, the Boltzmann equation reads
\begin{equation}\label{Boltzmann}
\partial_t f(t,v) = Q(f,f)(t,v).
\end{equation}
The collision operator $Q(f,f)$ is defined by
\begin{equation}\label{Q(f,f)}
Q(f,f)(t,v)= \int_{\mathbb{R}^d} \int_{S^d-1} \left( f' \, f'_* - f \, f_* \right) |v-v_*|^\gamma b_B\left(\hat{u}\cdot\sigma\right) \, \mathrm{d} \sigma \, \mathrm{d} v_*,
\end{equation}
with the standard abbreviations $f_*:=f(t,v_*)$, $f':=f(t,v')$, $f'_*:=f(t,v'_*)$. For a pair of particles, vectors $v', v'_*$  denote pre-collisional velocities,  while vectors $v, v_*$ denote their post-collisional velocities. Local momentum and energy are conserved, i.e.
\begin{align*}
v'+v'_* &= v+v_*\\
|v'|^2 + |v_*'|^2 &= |v|^2 + |v|^2_*.
\end{align*}
Thus by introducing a parameter $\sigma \in S^{d-1}$, the collision laws can be expressed as
\begin{align}\label{B-velocities}
v' & = \frac{v+v_*}{2} + \frac{|v-v_*|}{2}\sigma,\\
v'_* & = \frac{v+v_*}{2} - \frac{|v-v_*|}{2}\sigma. \nonumber
\end{align}
The unit vector $\sigma \in S^{d-1}$ has the direction of the relative velocity $u' = v'-v'_*$, while the normalization of the relative velocity $u=v-v_*$ is denoted by  $\hat{u}:=\frac{u}{|u|}$. The angle between these two directions, denoted by $\theta$, is called the scattering angle and it satisfies $\hat{u} \cdot \sigma = \cos\theta$.

Due to physical considerations, the parameter $\gamma$ is a number in the range $(-d, 1]$. In this paper we consider the Maxwell molecules model, which corresponds to
\begin{align}
\gamma = 0.
\end{align}

  The angular kernel $b_B(\hat{u}\cdot \sigma)= b_B(\cos\theta)$ is a non-negative function that encodes the likelihood of collisions between particles. It has a singularity for $\sigma$ that satisfies $\hat{u}\cdot \sigma =1$, i.e.  $\theta =0$, which may or may not be integrable in $\sigma \in S^{d-1}$. Its integrability is often referred to as the angular cutoff, while its non-integrability is referred to as the non-cutoff case. In this paper we assume that
\begin{equation}\label{Boltzmann cross section}
\int_0^\pi b_B(\cos\theta) \, \sin^\beta\theta \sin^{d-2}\theta \, d\theta < \infty, \ \text{for some} \ \beta \in [0,2].
\end{equation}
The case $\beta = 0$ corresponds to $b_B (\hat{u} \cdot \sigma)$ being integrable in $\sigma \in S^{d-1}$, i.e. it corresponds to the cutoff case. When $\beta >0$, then the angular kernel $b_B$ is allowed to have $\beta$ more degrees of singularity compared to the cutoff case.

In particular,  in the case of inverse  power-law potentials for the Maxwell molecules, the interaction potential in 3 dimensions is of the form $V(r)=r^{-4}$. Then a nonintegrable singularity of the function $b_B$ is known
\begin{equation*}
b_B(\cos\theta) \, \sin\theta \sim \theta^{-\frac{3}{2}}, \quad \theta \rightarrow 0.
\end{equation*}
Therefore, $\beta$ should satisfy $\beta>\frac{1}{2}$.

\subsection{Weak formulation of the collision operator}
Since the Jacobian of the pre to post collision transformation is unit and due to the symmetries of the kernel, for any sufficiently smooth test function $\phi(v)$, the weak formulation of the collision operator $Q(f,f)$ reads
\begin{multline}\label{weak formulation B}
\int_{\mathbb{R}^d}  Q(f,f) \, \phi(v) \, \mathrm{d} v \\ = \frac{1}{2} \int_{\mathbb{R}^{2d}} f f_* |v-v_*|^\gamma \int_{S^{d-1}}  \left( \phi(v') + \phi(v'_*) - \phi(v) - \phi(v_*) \right) b_B(\hat{u}\cdot \sigma) \, \mathrm{d} \sigma \, \mathrm{d} v_* \, \mathrm{d} v.
\end{multline}

\subsection{Weak solutions to the Boltzmann equation}
We recall the definition of a weak solution to the Cauchy problem for the Boltzmann equation
\begin{equation}\label{Boltzmann Cauchy}
\left\{\begin{split} \partial_t f(t,v)&=Q(f,f)(t,v) \quad t\in \mathbb{R}_+, \ v \in \mathbb{R}^d,\\f(0,v)&=f_0(v), \end{split}\right.
\end{equation}
whose existence in three dimensions and for the angular kernel \eqref{Boltzmann cross section} with $\beta \in [0,2]$ is proved in \cite{ar81, vi98}.

\begin{definition}\label{defn-weak}
Let $f_0 \geq 0$ be a function defined in $\mathbb{R}^d$ with finite mass, energy and entropy
\begin{align}\label{Boltzmann initial data}
\int_{\mathbb{R}^d} f_0(v) \, \left( 1 + |v|^2 + \log(1+f_0(v))\right) \, dv < +\infty.
\end{align}
Then we say $f$ is a {\it weak solution} to the Cauchy problem \eqref{Boltzmann Cauchy} if it satisfies the following conditions
\begin{itemize}
\item $f\geq 0, \; f \in C(\mathbb{R}^+; \mathcal{D}'(\mathbb{R}^d)) \, \cap \, f \in L^1([0,T]; L^1_{2+\gamma})$ \vspace{5pt}
\item $f(0,v) = f_0(v)$ \vspace{5pt}
\item $\forall t\geq 0$: \; $\int f(t,v) \psi(v) dv = \int f_0(v) \psi(v) dv$, for $\psi(v) = 1, v_1, ..., v_d, |v|^2$ \vspace{5pt}
\item $f(t,\cdot) \in L\log L$ and $ \forall t\geq 0: \, \int f(t,v) \log f(t,v) dv \leq \int f_0(v) \log f_0 dv$ \vspace{5pt}
\item $\forall \phi(t,v) \in C^1(\mathbb{R}^+, C^\infty_0(\mathbb{R}^3))$, $\forall t \geq 0$ we have that
\begin{align*}
\int_{\mathbb{R}^d} f(t,v) \phi(t,v) dv \,  - \, \int_{\mathbb{R}^d} f_0(v) \phi(0,v) dv
\,& - \, \int_0^t d\tau \int_{\mathbb{R}^d} f(\tau, v) \partial_\tau \phi(\tau,v) dv\\
& = \int_0^t d\tau \int_{\mathbb{R}^d} Q(f,f)(\tau,v) \phi(\tau,v) dv.
\end{align*}
\end{itemize}

\end{definition}

\section{The main results}\label{sec-main}

Our main result establishes propagation of stretched exponential moments for the Kac equation and for the Boltzmann equation corresponding to Maxwell molecules.
\begin{theorem}\label{thm}
Suppose initial datum $f_0 \geq 0$ has finite mass, energy and entropy, i.e. \eqref{Kac initial data} in case of the Kac equation and \eqref{Boltzmann initial data} in the case of the Boltzmann equation.
\begin{itemize}
\item[(a)] Kac equation: Let $f(t,v)$ be an associated weak solution to the Cauchy problem \eqref{Cauchy1}, with \eqref{K(f)} and with the angular kernel satisfying \eqref{cross section} with $\kappa \in [0,2]$. If
\begin{equation}\label{order K}
s \leq  \displaystyle\frac{4}{2+\kappa},
\end{equation}
then for every $\alpha_0>0$ there exists  $0<\alpha \leq \alpha_0$ and a constant $C>0$ (depending only on the initial data and $\kappa$) so that
\begin{align}\label{K thm}
&\mbox{if} \,\, \int_{\mathbb{R}} f_0(v) e^{\alpha_0 \langle v \rangle^s} \mathrm{d} v \leq M_0 < \infty, \nonumber \\
& \qquad  \mbox{then} \,\, \int_{\mathbb{R}} f(t,v) e^{\alpha \langle v \rangle^s} \mathrm{d} v\leq C, \,\,\, \forall t\geq0.
\end{align}
\item[(b)] Boltzmann equation for Maxwell molecules: Let $f(t,v)$ be an associated weak solution to the Cauchy problem \eqref{Boltzmann Cauchy} with the angular kernel satisfying \eqref{Boltzmann cross section} with $\beta \in [0,2]$. If
\begin{equation}\label{order B}
s \leq  \displaystyle\frac{4}{2+\beta},
\end{equation}
then for every $\alpha_0>0$ there exists  $0<\alpha \leq \alpha_0$ and a constant $C>0$ (depending only on the initial data and $\beta$) so that
\begin{align}\label{B thm}
& \mbox{if} \,\, \int_{\mathbb{R}^d} f_0(v) e^{\alpha_0 \langle v \rangle^s} \mathrm{d} v \leq M_0 < \infty, \nonumber \\
& \qquad  \mbox{then} \,\, \int_{\mathbb{R}^d} f(t,v) e^{\alpha \langle v \rangle^s} \mathrm{d} v\leq C, \,\,\, \forall t\geq0.
\end{align}
\end{itemize}
\end{theorem}

\begin{remark}
We make several remarks about this result.
\begin{enumerate}
\item[(i)] The order $s$ of the stretched exponential moment that propagates in time depends on the singularity rate of the angular kernel. According to \eqref{order K} and \eqref{order B}, the more singular the kernel is, the smaller the $s$ is.
\item[(ii)] The Maxwellian moment $s=2$ can be reached only in the cutoff case i.e. $\kappa = 0$ for the Kac equation or $\beta = 0$ for the Boltzmann equation.
\item[(iii)] The cutoff Kac equation was studied in \cite{de93}, where propagation of moments of order $s=1$ and $s=2$ was proved. We extend this result by allowing $s\in[0,2]$, and by considering the non-cutoff kernels too.
\item[(iv)] The cutoff Boltzmann equation for Maxwell molecules was studied in \cite{bo84, bo88}, where propagation of Maxwellian moments ($s=2$) is proved. We extend this result by allowing $s\in [0,2]$.
\item[(v)] The non-cutoff Boltzmann equation for hard potentials $\gamma >0$ was studied in \cite{lumo12} where generation of stretched exponential moments of order $s=\gamma$ was proved, and \cite{algapata15} where propagation of stretched exponential moments was proved depending on the singularity rate of the angular kernel. We extend the work of \cite{algapata15} to include the case $\gamma =0$.
\end{enumerate}
\end{remark}

\section{Mittag-Leffler moments}\label{sec-ml} In this section, we recall the definition of Mittag-Leffler moments, first introduced in \cite{algapata15}. They are a generalization of stretched exponential moments, and they are convenient for the study of exponential decay properties of a function $f$. Namely, these moments are the $L^1$ norms weighted with Mittag-Leffler functions which asymptotically behave like exponentials. More precisely, a Mittag-Leffler function with parameter $a>0$ is defined by
\begin{equation*}
\mathcal{E}_a (x):=\sum_{q=0}^{\infty} \frac{x^q}{\Gamma(a q+1)}, \qquad a>0, x \in \mathbb{R}.
\end{equation*}
Note that $\mathcal{E}_1 (x)$ is simply the Maclaurin series of $e^x$, while it is well-known that for $a>0$
\begin{equation*}
\mathcal{E}_a (x) \sim e^{x^{1/a}}, \quad \text{as} \ x \rightarrow + \infty.
\end{equation*}
Therefore,
\begin{equation*}
 e^{\alpha \langle v \rangle^s} \sim \mathcal{E}_{2/s} (\alpha^{2/s}  \langle v \rangle^2) = \sum_{q=0}^{\infty} \frac{\alpha^{ \frac{2q}{s}} }{\Gamma(\frac{2}{s} q+1)} \langle v \rangle^{2q} , \quad \text{when} \ v \rightarrow + \infty.
\end{equation*}
This motivated the definition of Mittag-Leffler moment \cite{algapata15}
\begin{definition}
The Mittag-Leffler moment of a rate $\alpha>0$ and an order $s>0$ is defined via

\begin{equation*}
\mathcal{M}_{\alpha, s}(t):=\int_{\mathbb{R}^d} f(t, v)  \, \mathcal{E}_{2/s} (\alpha^{2/s}  \langle v \rangle^s)\, \mathrm{d} v = \sum_{q=0}^{\infty} \frac{\alpha^{ \frac{2q}{s}} }{\Gamma(\frac{2}{s} q+1)} m_{2q}(t)
\end{equation*}
for any $t\geq 0$.
\end{definition}
\begin{remark}\label{exp=ml}
Due to the asymptotic behavior of Mittag-Leffler functions, the finiteness of the stretched exponential moment $M_{\alpha,s}(t)$ at any time $t>0$ is equivalent to the finiteness of the corresponding Mittag-Leffler moment $\mathcal{M}_{\alpha, s}(t)$.
\end{remark}

\section{Angular averaging lemmas with cancellation}\label{sec-lem}
Before proving Theorem \ref{thm}, we provide an estimate of the angular part of the weak formulation \eqref{weak formulation K} and \eqref{weak formulation B} when the test function is a monomial $\phi(v)=\langle v \rangle^{2q}$. These bounds will be later used to derive a differential inequality for polynomial moment in Lemma \ref{lemma polynomial moment ODE}.

\begin{lemma}\label{lemma theta integral}
Let $q\geq 2$.
\begin{itemize}
\item[(a)] Kac equation: Suppose that the angular kernel of the Kac equation $b_K$ satisfies the assumption \eqref{cross section}.  Then
\begin{multline*}
\int_{-\pi}^{\pi} \left( \langle v' \rangle^{2q}  + \langle v'_* \rangle^{2q} -  \langle v \rangle^{2q}  - \langle v_* \rangle^{2q} \right) \,  b_K(|\theta|)  \, \mathrm{d} \theta
\\ \leq- \frac{C_1}{2}  \left( \langle v \rangle^{2q} + \langle v_* \rangle^{2q}  \right)  + \frac{C_1}{2}  \left( \langle v \rangle^{2} \langle v_* \rangle^{2q-2} + \langle v \rangle^{2q-2} \langle v_* \rangle^{2}  \right) \\
+  C_1 q (q-1) \varepsilon_{\kappa, q} \langle v \rangle^2 \langle v_* \rangle^2 \left( \langle v \rangle^2 + \langle v_* \rangle^2 \right)^{q-2},
\end{multline*}
where
\begin{equation}\label{C1}
C_1 = \int_{-\pi}^{\pi} \sin^2 (2\theta) \,  b_K(|\theta|) \, \mathrm{d} \theta < \infty
\end{equation}
and
\begin{equation}\label{eps k}
\varepsilon_{\kappa, q} =  \frac{2}{C_1}\int_{-\pi}^{\pi} \sin^2 (2\theta) \, b_K(|\theta|) \int_0^1 t \left( 1-  \frac{t}{4} \sin^2 2\theta \right)^{q-2} \mathrm{d} t \, \mathrm{d} \theta \leq 1.
\end{equation}

\item[(b)] Boltzmann equation for Maxwell molecules: Suppose that the angular kernel $b_B$ satisfies the assumption \eqref{Boltzmann cross section}.
\begin{multline*}
\int_{S^{d-1}} \left( \langle v' \rangle^{2q}  + \langle v'_* \rangle^{2q} -  \langle v \rangle^{2q}  - \langle v_* \rangle^{2q} \right) \,  b_B(\hat{u}\cdot \sigma)  \, \mathrm{d} \sigma
\\ \leq - C_2  \left( \langle v \rangle^{2q} + \langle v_* \rangle^{2q}  \right)  + C_2 \left( \langle v \rangle^{2} \langle v_* \rangle^{2q-2} + \langle v \rangle^{2q-2} \langle v_* \rangle^{2}  \right) \\
+ C_2 q (q-1) \varepsilon_{\beta, q} \langle v \rangle^2 \langle v_* \rangle^2 \left( \langle v \rangle^2 + \langle v_* \rangle^2 \right)^{q-2},
\end{multline*}
where
\begin{equation}\label{C2}
C_2 = |S^{d-2}|  \int_{0}^{\pi} b_B(\cos\theta) \sin^d \theta \,   \mathrm{d} \theta < \infty
\end{equation}
and
\begin{equation}\label{eps B}
\varepsilon_{\beta, q} = \frac{2}{C_2}|S^{d-2}| \int_{0}^{\pi} \sin^d (\theta) \, b_B(\cos\theta) \int_0^1 t \left( 1-  \frac{t}{2} \sin^2 \theta \right)^{q-2} \mathrm{d} t \, \mathrm{d} \theta \leq 1.
\end{equation}
\end{itemize}
\end{lemma}

\begin{remark}\label{decay of b}
The sequences $\{\varepsilon_{\kappa, q}\}_q$ and $\{\varepsilon_{\beta, q}\}_q$  are decreasing to zero with a certain decay rate depending on the angular singularity rate $\kappa\in [0,2]$ in the case of the Kac equation and $\beta\in[0,2]$ in the case of the Boltzmann equation, \cite{lumo12},
\begin{align} \label{decay rate}
&\varepsilon_{\kappa, q} \;q^{1-\frac{\kappa}{2}} \rightarrow 0, \quad \mbox{as} \;\; q\rightarrow \infty, \\
&\varepsilon_{\beta, q} \;q^{1-\frac{\beta}{2}} \rightarrow 0, \quad \mbox{as} \;\; q\rightarrow \infty.
\end{align}
\end{remark}

\begin{proof}[Proof of Lemma \ref{lemma theta integral}]
The proof of part (b) can be found in \cite[Lemma 2.3]{algapata15}. Thus, here we provide only the proof of part (a).

If $E(\theta)$ denotes the following convex combination of particle energies
\begin{equation}\label{E(theta)}
\begin{split}
E(\theta) &:= \langle v \rangle^2 \cos^2\theta +  \langle v_* \rangle^{2} \sin^2\theta, \end{split}
\end{equation}
then, using the collision rules \eqref{velocities}, we obtain
\begin{equation}\label{velocities with E(theta)}
\begin{split}
\langle v' \rangle^{2} &= E(\theta) - 2 v v_* \sin \theta \cos \theta,\\
\langle v'_* \rangle^{2} &= E(\pi - \theta) + 2 v v_* \sin \theta \cos \theta.
\end{split}
\end{equation}
Taylor expansion of $\langle v' \rangle^{2q}$ around $E(\theta)$ up to the second order yields
\begin{align*}
\langle v' \rangle^{2q} &= \left( E(\theta) - 2 v v_* \sin \theta \cos \theta  \right)^q\\
                        &= E(\theta)^q - 2 q E(\theta)^{q-1} v v_* \sin \theta \cos \theta \\ & \qquad + 4 q (q-1) v^2 v_*^2 \sin^2\theta \cos^2\theta \int_0^1 (1-t) \left( E(\theta) - 2 \,t\, v v_* \sin \theta \cos \theta \right)^{q-2} \mathrm{d} t.
\end{align*}
Analogous expression can be written for $\langle v'_*\rangle$ as well.

The first order term in the above expression is an odd function in $\theta$, which nullifies by integration over the even domain $[-\pi, \pi]$. Therefore, we can write
\begin{equation*}
\int_{-\pi}^{\pi} \left( \langle v' \rangle^{2q}  + \langle v'_* \rangle^{2q} -  \langle v \rangle^{2q}  - \langle v_* \rangle^{2q} \right) \,  b_K(|\theta|)  \, \mathrm{d} \theta
= I_1 + I_2,
\end{equation*}
where
\begin{align*}
I_1 &= \int_{-\pi}^{\pi} \left( E(\theta)^q + E(\pi-\theta)^q - \langle v \rangle^{2q} -  \langle v_* \rangle^{2q} \right) b_K(|\theta|) \, \mathrm{d} \theta,\\
I_2 &=4 q (q-1) v^2 v_*^2 \int_{-\pi}^{\pi} \sin^2\theta \cos^2\theta  \, b_K(|\theta|) \\ & \qquad \times \int_0^1 (1-t) \left[ \left( E(\theta) - 2 \,t\, v v_* \sin \theta \cos \theta \right)^{q-2} + \left( E(\pi-\theta) + 2 \,t\, v v_* \sin \theta \cos \theta \right)^{q-2} \right]\mathrm{d} t \, \mathrm{d} \theta.
\end{align*}

We now proceed to estimate the terms $I_1$ and $I_2$ separately.

\subsubsection*{Term $I_1$}
 The term $I_1$ is estimated by an application of Lemma \ref{lemma convex binomial expansion estimate}. Indeed, for $t=\cos^2 \theta$, so that $1-t=\sin^2 \theta$, and $a=\langle v \rangle^2$, $b=\langle v_* \rangle^2$, recalling \eqref{E(theta)} we obtain:
\begin{align*}
I_1 & \leq  - 2 \int_{-\pi}^{\pi}  \cos^2 \theta \sin^2 \theta \left( \langle v \rangle^{2q} + \langle v_* \rangle^{2q}  \right) b_K(|\theta|) \, \mathrm{d} \theta \\
&  \quad \qquad +  2 \int_{-\pi}^{\pi}  \cos^2 \theta \sin^2 \theta \left( \langle v \rangle^{2} \langle v_* \rangle^{2q-2} + \langle v \rangle^{2q-2} \langle v_* \rangle^{2}  \right) b_K(|\theta|) \, \mathrm{d} \theta \\
& = - \frac{1}{2} \left( \langle v \rangle^{2q} + \langle v_* \rangle^{2q}  \right) \int_{-\pi}^{\pi} \sin^2 2\theta \,  b_K(|\theta|) \, \mathrm{d} \theta \\
&  \quad \qquad + \frac{1}{2} \left( \langle v \rangle^{2} \langle v_* \rangle^{2q-2} + \langle v \rangle^{2q-2} \langle v_* \rangle^{2}  \right) \int_{-\pi}^{\pi} \sin^2 2\theta \,  b_K(|\theta|) \, \mathrm{d} \theta.
\end{align*}
 Therefore, recalling \eqref{C1}, we have
\begin{equation*}
I_1 \leq - \frac{C_1}{2} \left( \langle v \rangle^{2q} + \langle v_* \rangle^{2q}  \right)  + \frac{C_1}{2}  \left( \langle v \rangle^{2} \langle v_* \rangle^{2q-2} + \langle v \rangle^{2q-2} \langle v_* \rangle^{2}  \right).
\end{equation*}
\subsubsection*{Term $I_2$} By Cauchy-Schwartz inequality, $- 2 \,t\, v v_* \sin \theta \cos \theta  \leq t \, E(\pi - \theta)$.
 Thus,
\begin{align*}
E(\theta) - 2 \,t\, v v_* \sin \theta \cos \theta
&\leq E(\theta) + t \, E(\pi - \theta)\\
&= \langle v \rangle^2 + \langle v_* \rangle^2 - (1-t) E(\pi - \theta) \\
& \leq \left( \langle v \rangle^2 + \langle v_* \rangle^2 \right)\left( 1- (1-t) \frac{\sin^2 2\theta}{4} \right),
\end{align*}
where the last inequality follows from
\begin{align*}
E(\pi - \theta)
&= \langle v \rangle^2 \sin^2\theta + \langle v_* \rangle^2 \cos^2\theta \\
& \geq \left( \langle v \rangle^2 + \langle v_* \rangle^2 \right) \min\left\{\sin^2\theta, \cos^2\theta\right\} \\
& \geq \left( \langle v \rangle^2 + \langle v_* \rangle^2 \right)\frac{\sin^2 2\theta}{4}.
\end{align*}
The same estimate holds for $E(\pi - \theta) + 2t v v_* \sin\theta \cos\theta$. Therefore, recalling \eqref{eps k} the definition of $\varepsilon_{\kappa, q}$, we have
\begin{equation*}
I_2 \leq 2 q (q-1) v^2 v_*^2 \left( \langle v \rangle^2 + \langle v_* \rangle^2 \right)^{q-2} \varepsilon_{\kappa, q}.
\end{equation*}
Adding estimates for terms $I_1$ and $I_2$ completes the proof of lemma.
\end{proof}

\section{Bounds on polynomial moments} \label{sec-pol}

\begin{lemma}\label{lemma polynomial moment ODE}
With assumptions and notations of Lemma \ref{lemma theta integral}, we have the following differential inequalities for a polynomial moment $m_{2q}$
\begin{itemize}
\item [(a)] In the case of the Kac equation,
\begin{equation}\label{polynomial moment ODE K}
m'_{2q} \leq - C_1 m_0 m_{2q} + C_1 m_2 m_{2q-2} + C_1 q (q-1) \, \varepsilon_{\kappa, q} \, \sum_{k=1}^{\lfloor \frac{q+1}{2} \rfloor} \left( \begin{matrix} q-2 \\ k -1\end{matrix} \right) m_{2k}m_{2q-2k}.
\end{equation}
\item [(b)] In the case of the Boltzmann equation for Maxwell molecules,
\begin{equation}\label{polynomial moment ODE B}
m'_{2q} \leq - C_2 m_0 m_{2q} + C_2 m_2 m_{2q-2} + C_2 \, q (q-1) \, \varepsilon_{\beta, q} \, \sum_{k=1}^{\lfloor \frac{q+1}{2} \rfloor} \left( \begin{matrix} q-2 \\ k -1\end{matrix} \right) m_{2k}m_{2q-2k}.
\end{equation}
\end{itemize}
\end{lemma}

\begin{proof}
Multiplying the Kac equation \eqref{Kac} with $\langle v \rangle^{2q}$ and integrating with respect to $v$, we obtain an equation for the polynomial moment $m_{2q}$
\begin{equation*}
m'_{2q} = \int_{\mathbb{R}} \langle v \rangle^{2q} \,  K(f,f)(t,v)  \, \mathrm{d} v.
\end{equation*}
Using the weak formulation \eqref{weak formulation K} one has
\begin{equation}
m'_{2q} = \int_{\mathbb{R}} \int_{\mathbb{R}} f f_*   \int_{-\pi}^{\pi} \left( \langle v' \rangle^{2q}  + \langle v'_* \rangle^{2q} -  \langle v \rangle^{2q}  - \langle v_* \rangle^{2q} \right) \, b_K(|\theta|)  \, \mathrm{d} \theta \, \mathrm{d} v_* \, \mathrm{d} v.
\end{equation}
Applying  Lemma \ref{lemma theta integral} and Lemma \ref{lemma polynomial inequality} yields
\begin{align*}
m'_{2q} &\leq \int_{\mathbb{R}} \int_{\mathbb{R}} f f_* \left[  - \frac{C_1}{2}  \left( \langle v \rangle^{2q} + \langle v_* \rangle^{2q}  \right)  + \frac{C_1}{2}  \left( \langle v \rangle^{2} \langle v_* \rangle^{2q-2} + \langle v \rangle^{2q-2} \langle v_* \rangle^{2}  \right) \right. \\ & \left.
  \hspace*{2.5cm} + \frac{C_1}{2}q (q-1) \, \varepsilon_{\kappa, q} \, v^2 v_*^2 \left( \langle v \rangle^2 + \langle v_* \rangle^2 \right)^{q-2} \right] \mathrm{d} v_* \, \mathrm{d} v\\
& \leq \int_{\mathbb{R}} \int_{\mathbb{R}} f f_* \left[  - \frac{C_1}{2}  \left( \langle v \rangle^{2q} + \langle v_* \rangle^{2q}  \right)  + \frac{C_1}{2}  \left( \langle v \rangle^{2} \langle v_* \rangle^{2q-2} + \langle v \rangle^{2q-2} \langle v_* \rangle^{2}  \right) \right. \\ &
  \hspace*{2.5cm} + \frac{C_1}{2} q (q-1) \, \varepsilon_{\kappa, q}  \, \langle v \rangle^2 \langle v_* \rangle^2  \\ & \hspace*{2cm}\left. \times \left( \sum_{k=0}^{k_{q-2}}  \left( \begin{matrix} q-2 \\ k \end{matrix} \right)  \left( \langle v \rangle^{2k} \langle v_* \rangle^{2 (q-2-k)} + \langle v \rangle^{2(q-2-k)}\langle v_* \rangle^{2k} \right) \right) \right] \mathrm{d} v_* \, \mathrm{d} v\\
&=- C_1 m_0 m_{2q} + C_1 m_2 m_{2q-2} +  C_1 q (q-1) \, \varepsilon_{\kappa, q} \, \sum_{k=0}^{k_{q-2}} \left( \begin{matrix} q-2 \\ k \end{matrix} \right) m_{2k+2}m_{2q-2k-2}.
\end{align*}
It remains to change index $k$ in the sum and the part (a) is proven. The proof of the part (b) can be done in an analogous way.
\end{proof}

\begin{lemma}[Propagation of polynomial moments ]\label{poly propagation}
Suppose $f$ is a weak solution to the Cauchy problem of either Kac equation \eqref{Cauchy1} or the Boltzmann equation for Maxwell molecules \eqref{Boltzmann Cauchy}. Then for every $q \geq 0$, we have
\begin{align}
m_{2q}(0) < \infty \qquad \Rightarrow \qquad m_{2q}(t) \leq C_q^*,
\end{align}
where the constant $C^*_q>0$, is uniform in time, and depends on  $q$ and the first $q$ moments of the initial data.
\end{lemma}

\begin{proof}
Applying the inequality \eqref{monotonicity of moments} to the differential inequality \eqref{polynomial moment ODE K} yields
\begin{align}\label{m'<m}
m_{2q}' \leq -C_1 m_0 m_{2q} + C_q m_{2q-2},
\end{align}
where $C_q = C_1 m_2(0) + C_1 q(q-1) \, 2^{q-2}$.
Therefore,
\begin{align}
m_{2q}(t) \leq \max \left\{m_{2q}(0), \, \frac{C_q }{C_1 m_0(0)} m_{2q-2}(t) \right\}.
\end{align}
Applying this inequality inductively, for an integer $q\in \mathbb{N}$, we have
\begin{align*}
m_{2q}(t)
& \leq \max\left\{ m_{2q}(0), \, \frac{C_q}{C_1 m_0(0)} m_{2q-2}(0), \, \frac{C_q }{C_1 m_0(0)} \frac{C_{q-1}}{C_1 m_0(0)} m_{2q-4}(t)\right\}\\
&....\\
&\leq \max\left\{ m_{2q}(0), \, \frac{C_q  m_{2q-2}(0)}{C_1 m_0(0)}, \, \frac{C_q C_{q-1}m_{2q-4}(0)}{(C_1 m_0(0))^2}  , \dots, \, \frac{C_q C_{q-1} \dots C_{2}m_{2}(t)}{ (C_1 m_0(0))^{q-1}}\right\}\\
&\leq \max\left\{ m_{2q}(0), \, \frac{C_q  m_{2q-2}(0)}{C_1 m_0(0)}, \, \frac{C_q C_{q-1}m_{2q-4}(0)}{(C_1 m_0(0))^2}  , \dots, \, \frac{C_q C_{q-1} \dots C_{2}m_{2}(0)}{ (C_1 m_0(0))^{q-1}}\right\}.
\end{align*}
Therefore, every even moment is bounded uniformly in time.

Moments whose order is not an even integer can be interpolated by even moments. For example, if $0<2q -2 \leq p \leq 2q$, then
$$m_p \leq m_{2q-2}^{\frac{2q-p}{2}} \, m_{2q}^{\frac{p-2q+2}{2}}. $$
Hence, polynomial moments of non-even orders are bounded uniformly in time as well.

\end{proof}

\begin{remark}\label{derivatives}
 We also note that derivatives of polynomial moments are uniformly bounded in time.  Namely, applying the inequality \eqref{monotonicity of moments} to the differential inequality \eqref{polynomial moment ODE K} yields
\begin{align}\label{m'<m}
m_{2q}'(t) \leq  C_q m_{2q-2}(t),
\end{align}
where $C_q = C_1 m_2(0) + 4q(q-1)\, \varepsilon_{\kappa, q} \, 2^{q-2}$. Lemma \ref{poly propagation} implies that $m_{2q-2}(t)$ is bounded uniformly in time by a constant $C^*_{q-1}$. Thus,
\begin{align}\label{m'}
m_{2q}'(t)  \leq C_q C^*_{q-1}.
\end{align}
\end{remark}

\section{Proof of Theorem \ref{thm}} \label{sec-pf1}
\begin{proof}[Proof of Theorem \ref{thm} (a)]  Recall from Remark \ref{exp=ml} that finiteness of the stretched exponential moment $M_{\alpha,s}(t)$ is equivalent to finiteness of the Mittag-Leffler moment of the same rate and order. Therefore, we set out to prove finiteness of Mittag-Leffler moment of order $s$ and rate $\alpha$ that will be determined later:
\begin{equation}\label{ML a}
\mathcal{M}_{\alpha,\frac{2}{a}}(t)= \sum_{q=0}^{\infty} \frac{\alpha^{a q} }{\Gamma( a q+1)} m_{2q}(t),
\end{equation}
where
$$ a = \frac{2}{s}>1.$$
The case $a=1$ corresponds to $s=2$ i.e. Maxwellian moments. Propagation of such moments can be esablished according to \eqref{order K} only in the cutoff case $\kappa=0$. This result (propagation of Maxwellian moments in the cutoff case) was alrady established in \cite{de93}. Thus, we here focus on the case when $a>1$.

The goal is to prove that partial sums of \eqref{ML a}
\begin{equation*}
E^n(t) := \sum_{q=0}^{n} \frac{\alpha^{a q} }{\Gamma( a q+1)} m_{2q}(t).
\end{equation*}
 are bounded uniformly in time and $n$.

From the differential inequality for polynomial moments \eqref{polynomial moment ODE K}, and by denoting
\begin{multline*}
S_0 = \sum_{q=0}^{q_0-1} \frac{ m'_{2q} \, \alpha^{a q} }{\Gamma( a q+1)},  \quad S_1 = \sum_{q=q_0}^{n} \frac{ m_{2q} \, \alpha^{a q} }{\Gamma( a q+1)}, \quad S_2 = \sum_{q=q_0}^{n} \frac{ m_{2q-2} \, \alpha^{a q} }{\Gamma( a q+1)}, \\ S_3= \sum_{q=q_0}^{n}  \frac{ q(q-1) \, \varepsilon_{\kappa, q} \, \alpha^{a q} }{\Gamma( a q+1)} \sum_{k=1}^{k_q} \left( \begin{matrix} q-2 \\ k -1\end{matrix} \right) m_{2k}m_{2q-2k},
\end{multline*}
for any $q_0\geq 3$,
we obtain the following  differential inequality for the partial sum $E^n$:
\begin{equation}\label{E^n_t}
\frac{\mathrm{d}}{\mathrm{d} t} E^n \leq S_0 - A_2 m_0 S_1 + A_2 m_2 S_2 + 4 S_3.
\end{equation}
We proceed to estimate each $S_i$, $i=0,1,2,3$ separately.

For later purposes,  we introduce a constant which will be an upper bound for the first $q_0-1$ polynomial moments and their derivatives. Let,
\begin{equation}\label{cq0}
c_{q_0} = \max_{q=1,\dots, q_0-1} \left\{ C_q^*, C_q C^*_{q-1} \right\},
\end{equation}
where $C_q^*$ is the constant from Lemma \ref{poly propagation}, and $C_q C_{q-1}^*$ is from Remark \ref{derivatives}. Then
\begin{align}\label{first q0}
m_{2q}(t) \leq c_{q_0} \quad \mbox{and}\quad  m_{2q}'(t) \leq c_{q_0}  \quad \mbox{for} \;\; q=1,2,...,q_0 -1.
\end{align}

\subsubsection*{Term $S_0$} Since the mass is conserved \eqref{CL}, i.e. $m_0'=0$, the first term in the sum $S_0$ is equal to zero.
Hence, by \eqref{first q0} we have
\begin{align*}
S_0 &=  \alpha^a \sum_{q=1}^{q_0 -1} \frac{ m'_{2q} \, \alpha^{a (q-1)} }{\Gamma( a q+1)}\\
& \leq c_{q_0} \alpha^a \sum_{q=1}^{q_0 -1} \frac{ \alpha^{a (q-1)} }{\Gamma( a q+1)}.
\end{align*}
Reindexing the sum and using the monotonicity of the Gamma function $\Gamma(a q+ a +1)\geq \Gamma(q+1)$, for $a>1$, $q=0,1,2,\dots$, we obtain
\begin{align}\label{S0}
S_0 & \leq c_{q_0} \alpha^a \sum_{q=0}^{q_0 -2} \frac{ \alpha^{a q} }{\Gamma( a q+ a + 1)} \nonumber \\
&\leq c_{q_0} \alpha^a \sum_{q=0}^{q_0 -2} \frac{  \left( \alpha^{a} \right)^q }{\Gamma( q+1)} \nonumber \\
& \leq c_{q_{0}}\alpha^a e^{\alpha^a} \leq 2 c_{q_{0}}\alpha^a,
\end{align}
for $\alpha$ small enough so that
\begin{equation}\label{choosing alpha}
e^{\alpha^a} \leq 2, \quad \text{or} \quad \alpha \leq (\text{ln}2)^{1/a}.
\end{equation}

\subsubsection*{Term $S_1$}
Using the bound \eqref{first q0} and parameter $\alpha$ chosen so that \eqref{choosing alpha} holds, we have
\begin{align}\label{S1}
S_1 &= E^n - m_0 - \sum_{q=1}^{q_0-1} \frac{m_{2q} \alpha^{aq}}{\Gamma(aq+1)} \nonumber\\
&\geq E^n - m_0 - 2 c_{q_0} \alpha^a.
\end{align}

\subsubsection*{Term $S_2$} Using again the monotonicity of the Gamma function, we obtain
\begin{align}\label{S2}
S_2 &= \alpha^a \sum_{q=q_0}^n \frac{m_{2q-2} \alpha^{a(q-1)}}{\Gamma(aq+1)} \nonumber\\
&\leq  \alpha^a \sum_{q=q_0}^n \frac{m_{2(q-1)} \alpha^{a(q-1)}}{\Gamma(a(q - 1) +1)} \nonumber \\
& \leq  \alpha^a E^{n-1} \leq  \alpha^a E^n.
\end{align}
\subsubsection*{Term $S_3$} Using the property of the Beta function $B(x,y)=\frac{\Gamma(x)\Gamma(y)}{\Gamma(x+y)}$, the term $S_3$ can be rearranged
\begin{align*}
S_3 &= \sum_{q=q_0}^{n} \frac{ q(q-1)   \, \varepsilon_{\kappa, q} \,\alpha^{a q}}{\Gamma(aq+1)}  \sum_{k=1}^{k_q} \left( \begin{matrix} q-2 \\ k -1\end{matrix} \right)
\frac{m_{2k}\alpha^{ak} }{\Gamma(ak+1)} \frac{ m_{2q-2k}\alpha^{a q-ak}}{\Gamma(aq-ak+1)} \\
& \hspace*{5cm}\times {B(ak+1, aq - a k+1)}\Gamma(aq+2)\\
& \leq  \sum_{q=q_0}^{n}  q(q-1)(aq+1) \, \varepsilon_{\kappa, q} \left(  \sum_{k=1}^{k_q} \left( \begin{matrix} q-2 \\ k -1\end{matrix} \right) B(ak+1, aq - a k+1) \right) \\ & \hspace*{5cm}\times\left(  \sum_{k=1}^{k_q} \frac{m_{2k}\alpha^{ak} }{\Gamma(ak+1)} \frac{ m_{2q-2k}\alpha^{a q-ak}}{\Gamma(aq-ak+1)} \right)\\
&\leq a C_a  \sum_{q=q_0}^{n} q^{2-a} \varepsilon_{\kappa, q}  \sum_{k=1}^{k_q} \frac{m_{2k}\alpha^{ak} }{\Gamma(ak+1)} \frac{ m_{2q-2k}\alpha^{a q-ak}}{\Gamma(aq-ak+1)},
\end{align*}
where the last estimate follows by an application of Lemma \ref{Beta function}.

Since by \eqref{order K} we have
$$
2 - a = 2- \frac{2}{s} \leq 1 - \frac{\kappa}{2},
$$
 Remark \ref{decay of b} implies that $\left\{  q^{2-a} \varepsilon_{\kappa, q} \right\}_q$  is a decreasing sequence and
\begin{equation}\label{decay in proof}
 q^{2-a} \varepsilon_{\kappa, q} \rightarrow 0, \quad \mbox{as} \,\, q\rightarrow \infty.
\end{equation}
If we denote $c_a = a C_a$, then the monotonicity of $\left\{  q^{2-a} \varepsilon_{\kappa, q} \right\}_q$ yields
\begin{align}\label{S3}
S_3 & \leq   c_a q_0^{2-a}  \varepsilon_{\kappa, q_0}  \sum_{q=q_0}^{n}  \sum_{k=1}^{\lfloor \frac{q+1}{2}\rfloor} \frac{m_{2k}\alpha^{ak} }{\Gamma(ak+1)} \frac{ m_{2q-2k}\alpha^{a q-ak}}{\Gamma(aq-ak+1)} \nonumber \\
& \leq  c_a q_0^{2-a}  \varepsilon_{\kappa, q_0}
\left( \sum_{k=1}^{\lfloor \frac{n+1}{2}\rfloor} \frac{m_{2k}\alpha^{ak} }{\Gamma(ak+1)}\right)
\left( \sum_{\ell=1}^{n-1} \frac{ m_{2\ell}\alpha^{a \ell}}{\Gamma(a\ell+1)}\right) \nonumber \\
& \leq c_a q_0^{2-a}  \varepsilon_{\kappa, q_0} E^n E^n
\end{align}

Going back to \eqref{E^n_t} and applying the bounds \eqref{S0}, \eqref{S1}, \eqref{S2}, \eqref{S3} we obtain a differential inequality for the partial sum $E_n$
\begin{multline*}
\frac{\mathrm{d}}{\mathrm{d} t} E^n(t) \leq - A_2 m_0(t) E^n(t) + c_{q_0} \alpha^a + A_2 m_0(t)^2  \\
+ A_2 m_0(t) c_{q_0} \alpha^a   + A_2 m_2(t) \alpha^a E^n(t) + 4 c_a q_0^{2-a} \varepsilon_{\kappa, q_0}  (E^n(t))^2.
\end{multline*}
Due to the conservation of mass, i.e. $m_0(t) = m_0(0)$, and the dissipation of energy, i.e. $m_2(t) \leq m_2(0)$ for the weak solution $f$, we have
\begin{multline}\label{ODE E^n}
\frac{\mathrm{d}}{\mathrm{d} t} E^n(t) \leq - A_2 m_0(0) E^n(t) + c_{q_0} \alpha^a + A_2 m_0(0)^2  \\
+ A_2 m_0(0) c_{q_0} \alpha^a   + A_2 m_2(0) \alpha^a E^n(t) + 4 c_a q_0^{2-a} \varepsilon_{\kappa, q_0}  (E^n(t))^2.
\end{multline}

To show that such $E^n(t)$ is uniformly bounded in time and $n$, we define
\begin{equation*}
T_n := \sup \left\{ t \geq 0: E^n(\tau) < 4 M_0, \forall \tau \in [0,t) \right\},
\end{equation*}
where $M_0$ is the bound on the initial data in \eqref{thm}, with the goal of proving that $T_n = \infty$ for all $n\in\mathbb{N}$.

The number $T_n$ is well-defined and positive. Indeed, since $\alpha<\alpha_0$, at time $t=0$ we have
\begin{align*}
E^n(0) &=  \sum_{q=0}^{n} \frac{\alpha^{a q} }{\Gamma( a q+1)} m_{2q}(0)\\
&<  \sum_{q=0}^{\infty} \frac{\alpha_0^{a q} }{\Gamma( a q+1)} m_{2q}(0) \\
&= \mathcal{M}_{\alpha_0, \frac{2}{a}}(0) < 4 M_0,
\end{align*}
uniformly in $n$, by \eqref{thm}. Since $E^n(t)$ are continuous functions of $t$, $E^n(t)<4 M_0$ for $t$ on some positive time interval $[0,t_n)$, $t_n >0$. Therefore, $T_n>0$.

Also, since $E^n(t) \leq M_0$ holds on the time interval $[0,T_n]$, from \eqref{ODE E^n} we obtain the following differential inequality
\begin{multline*}
\frac{\mathrm{d}}{\mathrm{d} t} E^n(t) \leq - A_2 m_0(0) E^n(t) + c_{q_0} \alpha^a + A_2 m_0(0)^2 \\
+ A_2 m_0(0) c_{q_0} \alpha^a   + A_2 m_2(0) \alpha^a M_0 + 4 c_a q_0^{2-a} \varepsilon_{\kappa, q_0}  (M_0)^2.
\end{multline*}
We conclude that
\begin{align}\label{En ode}
E^n(t) &\leq M_0 + \frac{c_{q_0} \alpha^a + A_2 m_0(0)^2 + A_2 m_0(0) c_{q_0} \alpha^a   + A_2 m_2(0) \alpha^a M_0 + 4 c_a q_0^{2-a} \varepsilon_{\kappa, q_0}  (M_0)^2 }{A_2 m_0(0) } \nonumber \\
& = M_0 + m_0(0) + \alpha^a \left( \frac{c_{q_0}}{A_2 m_0(0)} + c_{q_0} + \frac{m_2(0) M_0}{m_0(0)} \right) + \frac{4 c_a q_0^{2-a} \varepsilon_{\kappa, q_0} (M_0)^2 }{A_2 m_0(0)}.
\end{align}
First, since $q_0^{2-a} \varepsilon_{\kappa, q_0} $ converges to zero as $q_0$ tends to infinity, we can choose $q_0$ large enough so that
\begin{equation}\label{one}
\frac{4 c_a q_0^{2-a} \varepsilon_{\kappa, q_0} (M_0)^2 }{A_2 m_0(0)} \leq \frac{M_0}{2}.
\end{equation}
Then, we choose $\alpha$ sufficiently small so that
\begin{equation}\label{two}
\alpha^a \, \left(\frac{c_{q_0}}{A_2 m_0(0)} + c_{q_0} + \frac{m_2(0) M_0}{m_0(0)}\right) \leq \frac{M_0}{2}.
\end{equation}
Therefore, applying estimates \eqref{one}, \eqref{two} and $m_0(0) \leq M_0$ to the differential inequality \eqref{En ode} yields
\begin{equation*}
E^n(t) \leq 3 M_0 < 4 M_0,
\end{equation*}
for any $t \in [0,T_n]$. Therefore, the strict inequality $E^n < 4 M_0$ holds on the closed interval $[0,T_n]$ for each $n$. But, since $E^n(t)$ is continuous function in $t$, the inequality $E^n(t) < 4 M_0$ holds on a slightly larger interval $[0, T_n + \mu)$, $\mu>0$. This contradicts definition of $T_n$ unless $T_n = + \infty$ for all $n$. Therefore,
\begin{equation*}
E^n(t) < 4 M_0 \quad \text{for any} \ t \in [0,+\infty) \ \text{and for all} \ n \in \mathbb{N}.
\end{equation*}
Hence, letting $n\rightarrow \infty$, we conclude that $\mathcal{M}_{\alpha, s}(t) < 4 M_0$ for all $t\geq 0$.
\vspace{10pt}

Part (b) of Theorem \eqref{thm} can be proved completely analogously to the proof of part (a). This is due to the similarity of the differential inequalities for polynomial moments \eqref{polynomial moment ODE K} and  \eqref{polynomial moment ODE B}. In addition, according to \eqref{decay rate} the decay rate of the sequences $\varepsilon_k,q$ and $\varepsilon_\beta, q$  depends in the exact same way on the the singularity rate of the angular kernel ($\kappa$ for the Kac equation and $\beta$ for the Boltzmann equation).
\end{proof}

\appendix
\section{Auxiliary results}\label{sec-app}

\begin{lemma}\label{lemma convex binomial expansion estimate}
Let $a, b \geq 0$, $t\in [0,1]$ and $p\in (0,1] \cup [2, \infty)$. Then
\begin{multline*}
\left( t a + (1-t) b \right)^p + \left( (1-t) a + t b \right)^p -a^p - b^p \\ \leq -2 t (1-t) (a^p + b^p) + 2 t (1-t) (a b^{p-1}+ a^{p-1}b).
\end{multline*}
\end{lemma}

\begin{lemma}\label{lemma polynomial inequality}
Assume $p>1$ and let $k_p=\lfloor \frac{p+1}{2} \rfloor$. Then, for all $x,y>0$ the following inequality holds:
\begin{equation*}
(x+y)^p \leq  \sum_{k=0}^{k_p}  \left( \begin{matrix} p \\ k \end{matrix} \right)  \left( x^k y^{p-k} + x^{p-k}y^k \right).
\end{equation*}
\end{lemma}

\begin{lemma}\label{lemma polynomial inequality 2}
Let $b\leq a \leq \frac{p}{2}$. Then for any $x, y \geq 0$
\begin{equation*}
x^a y^{p-a} + x^{p-a} y^a \leq x^b y^{p-b} + x^{p-b} y^b.
\end{equation*}
\end{lemma}

\begin{remark}
The above lemma is useful for comparing products of moments whose total homogeneity is the same. Namely,
\begin{equation}\label{monotonicity of moments}
m_0 \, m_{p} \leq m_1 \, m_{p-1} \leq m_2 \, m_{p-2} \leq ... \leq  m_{\lfloor \tfrac{p}{2}\rfloor} \, m_{p- \lfloor \tfrac{p}{2} \rfloor}.
\end{equation}
\end{remark}

\begin{lemma}\label{Beta function}
Let $p\geq 3$ and  $k_p=\lfloor \frac{p+1}{2} \rfloor$. Then for any $a>1$ we have
\begin{equation*}
 \sum_{k=1}^{k_p} \left( \begin{matrix} p-2 \\ k -1\end{matrix} \right) B(ak+1, a(p -  k)+1) \leq C_a q^{-(1+a)},
\end{equation*}
where the constant $C_a$ depends only on $a$.
\end{lemma}

\section*{Acknowledgments}
Authors are grateful to Laurent Desvillettes for suggesting us to study the problem considered in this paper. We also thank Ricardo J. Alonso for discussions on the subject. We would like to thank Irene M. Gamba and Nata\v{s}a Pavlovi\'{c} for their valuable remarks. The work of M. P.-\v C. was partially supported by the NSF grant NSF-DMS-RNMS-1107465 and by Project No. ON174016 of the Serbian Ministry of Education, Science and Technological Development.  The work of  M.T. has been supported by NSF grants DMS-1413064, NSF-DMS-RNMS-1107465 and DMS-1516228. Support from the Institute of Computational Engineering and Sciences (ICES) at the University of Texas Austin is gratefully acknowledged.

\medskip

\medskip

\end{document}